\documentclass[12pt]{amsart}
\usepackage{amsmath, amssymb, xcolor}
\usepackage[hidelinks]{hyperref}
\newtheorem{theorem}{\sc Theorem}[section]
\newtheorem{lemma}[theorem]{\sc Lemma}

\usepackage{abstract} % Permite maior personalização

\title[Chernikov-by-nilpotent groups]{On Chernikov-by-nilpotent groups}

\author[M. Capasso]{Martina Capasso } 
\address{Martina Capasso: Dipartimento di Matematica e Applicazioni, Università di Napoli Federico II, Complesso Universitario Monte S. Angelo, Naples (Italy)}
\email{martina.capasso2@unina.it}

\author[L. Lancellotti]{Liliana Lancellotti} 
\address{Liliana Lancellotti: Dipartimento di Matematica e Applicazioni, Università di Napoli Federico II, Complesso Universitario Monte S. Angelo, Naples (Italy)}
\email{liliana.lancellotti@unina.it}

\author[P. Shumyatsky]{Pavel Shumyatsky} 
\address{Pavel Shumyatsky: Department of Mathematics, University of Brasilia, Brasilia, DF, Brazil}
\email{pavel@unb.br}

\thanks{The first and the second authors are members of GNSAGA (INdAM) and  of AGTA – Advances in Group Theory and Applications (\href{www.advgrouptheory.com}{www.advgrouptheory.com}). The work of the third author was supported by FAPDF and CNPq. 
The first and second authors wish to thank the University of Brasília for its kind hospitality during the visit in which this work was prepared.}


\begin{document}

\maketitle
\begin{center}
    \emph{In memory of Francesco de Giovanni}
\end{center}

\medskip \medskip

\begin{abstract} Let $\gamma_k=[x_1,\dots,x_k]$ be the $k$-th lower central group-word. Given a group $G$, we write $X_k(G)$ for the set of $\gamma_k$-values and $\gamma_k(G)$ for the $k$-th term of the lower central of $G$. This paper deals with groups in which $\langle g^{X_k(G)} \rangle$ is a Chernikov group of size at most $(m,n)$ for all $g\in G$. The main result is that $\gamma_{k+1}(G)$ is a Chernikov group and its size is $(k,m,n)$-bounded. 
\end{abstract}

\section{Introduction}
Let $G$ be a group. If $X$ and $Y$ are non-empty subsets of $G$, we write $X^Y$ to denote the set $\{y^{-1}xy : x \in X , y \in Y \}$. Given $ k \geq 1$ and elements $x_1, \dots , x_k \in G$, we write $[x_1, \dots, x_k]$ for the left-normed commutator $[ \dots [[x_1, x_2], x_3] \dots , x_k]$. Elements that can be written as $[x_1, \dots, x_k]$ for suitable $x_1, \dots, x_k \in G$ will be called \emph{$\gamma_k$-values}, and the set of all $\gamma_k$-values in $G$ will be denoted by $X_k(G)$. Of course, the subgroup generated by the $\gamma_k$-values is the $k$-th term of the lower central series of $G$, usually denoted by $\gamma_k(G)$.

Of course, if the conjugacy class  $x^G$ of an element $x \in G$ is finite, we have $|x^G|= [G : C_G (x)]$. A group is said to be an \emph{$FC$-group} if $x^G$ is finite for every $x\in G$ and a \emph{$BFC$-group} if there is an integer n such that $|x^G|\leq n$ for every $x\in G$.
One of the most famous of B. H. Neumann’s theorems says that in a $BFC$-group the commutator subgroup is finite \cite{neumann}. Later J. Wiegold showed that if $|x^G| \leq n$ for every $x \in G$, then the order of $G'$ is bounded by a number depending only on $n$. Moreover, Wiegold found a first explicit bound for the order of $G'$ \cite{wie}, and the best known was obtained in \cite{gura} (see also \cite{neuvl,sesha}). \par
 
In recent years, Neumann's theorem has been extended in several directions. In particular, groups in which conjugacy classes containing commutators are bounded were treated in \cite{dms, dieshu}. Several authors have considered so-called \emph{$BCC$-groups}, that are groups in which $G/C_G(\langle x^G\rangle)$ is a Chernikov group of uniformly bounded size for every $x \in G$ (see \cite{DFAdGM2021}). 

Recall that a group $G$ is \emph{Chernikov} if it contains a radicable abelian normal subgroup $R$ (the finite residual) which is the direct product of a finite number $m(G)$ of groups of Prüfer type, such that the factor group $G/R$ is finite, of order $n(G)$ say. In general a group $G$ is called \emph{radicable} if the equation $x^n=a$ has a solution in $G$ for every positive integer $n$ and every $a \in G$.
By a deep result obtained independently by Shunkov \cite{Shunkov} and Kegel and Wehrfritz \cite{Kegel}, Chernikov groups are precisely the locally finite groups satisfying the minimal condition on subgroups, that is, any non-empty set of subgroups possesses a minimal subgroup. Furthermore, if $G$ is a Chernikov group with $m=m(G)$ and $n=n(G)$, the pair $(m,n)$ is usually called the \emph{size} of $G$. Of course, a Chernikov group $G$ is finite if and only if $m(G)=0$, and it is a radicable abelian group if and only if $n(G)=1$. 

It is well-known that the class of Chernikov groups is closed with respect to subgroups, homomorphic images and extensions. In particular, if $G$ is a group and $N$ is a normal subgroup of $G$ such that $N$ is a Chernikov group of size $(m,n)$ and $G/N$ is a Chernikov group of size $(m',n')$, then $G$ is a Chernikov group of size at most $((m+m'+n+n')^2, nn')$ and so bounded in terms of $m,n,m'$ and $n'.$ 

 Using the concept of verbal conjugacy classes introduced in \cite{FdGS2002}, an extension of Neumann’s theorem was obtained in \cite{DDMP2021}. It was shown that if $k,n$ are positive integers and $G$ is a group in which $|g^{X_k(G)}| \leq n$ for any $g\in G$, then $\gamma_{k+1}(G)$ has finite $(k, n)$-bounded order.

Throughout the article, we use the expression “$(a, b, \dots )$-bounded” to mean that a quantity is finite and bounded by a certain number depending only on the parameters $a, b, \dots$. 

In this paper we consider groups in which the subgroup $\langle g^{X_k(G)} \rangle$ is a Chernikov group of bounded size for every $g\in G$. We establish the following theorem. 

\begin{theorem}
\label{main}
Let $k,m,n$ be non-negative integers (with $k,n \neq0$), and let $G$ be a group such that $\langle g^{X_k{(G)}} \rangle$ is a Chernikov group of size at most $(m,n)$ for every $g\in G$. Then $\gamma_{k+1}(G)$ is a Chernikov group and its size is $(k,m,n)$-bounded. 
\end{theorem}

Obviously, this is an extension of the aforementioned result from \cite{DDMP2021}, which can be recovered from Theorem \ref{main} as a particular case where $m=0$.

\section{Proofs}

We start the section with the following well-known lemmas. The interested reader can find their proofs for example in \cite{MES2013}.

\begin{lemma}
\label{C1}
In a periodic nilpotent group $G$ every radicable abelian subgroup $R$ is central.
\end{lemma}

Let $A$ be a group acting on a group $G$. As usual, $[G,A]$ denotes the
subgroup generated by all elements of the form $x^{-1}x^a$, where $x \in G$, $a \in A$. It is well-known that $[G,A]$ is an $A$-invariant normal subgroup of $G$. \par
For elements $x,y$ in a group $G$, we write $[x, {}_{{l}}y]$ to denote the long commutator $[x,y, \dots, y]$, where $y$ is repeated $l$ times. If $N$ is a subgroup of $G$, the symbol $[N, y]$ stands for the subgroup generated by the commutators $[x, y]$, where $x$ ranges over $N$. For $l \geq 2$, define by induction $[N, {}_{{l}}y]= [[N, {}_{{l-1}}y],y]$. Note that if $N$ is an abelian normal subgroup, $[N, {}_{{l}}y]$ is precisely the set of commutators $[x, {}_{{l}}y]$, where $x \in N$. 

\begin{lemma}
\label{C3}
Let $A$ be a periodic group acting on a periodic radicable abelian group $G$. Then $[G, A , A] =[G,A]$.
\end{lemma}

\begin{lemma}
\label{L13}
Let $A$ be a finite group acting on a periodic radicable abelian group $G$. Then $[G,A]$ is radicable.
\end{lemma}

The following results will be useful later on.

\begin{lemma}
\label{L14}
Let $G$ be a periodic group, and let $N$ be a normal subgroup of $G$. Assume that $[N,x]$ is a Chernikov group for every $x \in G$, and let $R_x$ be its radicable part. Then the subgroup $R$ generated by the subgroups $R_x$ is a radicable abelian normal subgroup of $G$.
\end{lemma}

\begin{proof}
Clearly, for every $x\in G$, the subgroup $R_x$ is normal in $N$. Therefore $R$ is normal as well. Moreover, for every $x,y \in G$, the product $R_xR_y$ is nilpotent of class at most two and so abelian by Lemma \ref{C1}. Hence, $R$ is abelian as well and the lemma follows.
\end{proof}

Recall that a key property of radicable abelian groups is that any radicable subgroup of an abelian group $A$ is a direct factor of $A$.

\begin{lemma}
\label{P8}
Let $m$ be a non-negative integer, and let $G$ be a periodic nilpotent group acting on a periodic radicable abelian group $R$. Suppose that for each $x \in G$ the subgroup $[R,x]$ is a direct product of at most $m$ Prüfer subgroups. Then there exists an integer $h=h(m)$ depending only on $m$ such that the subgroup $[R,G]$ is a direct product of at most $h$ Prüfer subgroups.
\end{lemma}

\begin{proof}
For each $x\in G$, let $m_x$ be the number of Prüfer direct factors of the subgroup $[R,x]$. Define $m(G,R)$ to be the maximum of all the numbers $m_x$. The lemma will be proved by induction on $m(G,R)$.
\par 
Clearly, if $m(G,R)=0$, then $[R,G]$ is trivial. Assume that $m(G,R)\geq 1$ and the existence of $h(m-1)$ is established.

We may assume that $G$ is nontrivial and acts on $R$ faithfully. Choose a nontrivial element $z\in Z(G)$. Observe that $[R,z]$ is $G$-invariant. In a natural way the group $G$ acts on the quotient $R/[R,z]$, and thus we can consider the quantity $m(G,R/[R,z])$. If $m(G,R/[R,z]) < m(G,R)$, then the subgroup $[R,G]/[R,z]$ is a direct product of at most $h(m-1)$ Prüfer subgroups. Consequently, since $[R,z]$ is a direct factor of $[R,G]$, it follows that
$[R,G]$ itself is a direct product of at most $m+h(m-1)$ Prüfer subgroups.
\par
Therefore assume that $m(G,R/[R,z])=m(G,R)=m$. Then there exists an element $y\in G$ such that the subgroup $[R,y]/([R,y] \cap [R,z])$ is a direct product of exactly $m$ Prüfer subgroups.  On the other hand, $m_y\leq m$, and therefore $m_y=m$, which implies that the intersection $[R,y]\cap [R,z]$ does not contain any Pr\"ufer subgroups and hence is finite. Moreover, since $[R,z]$ is $y$-invariant and $[R,y]$ is $z$-invariant, we have $[R,y,z] \leq [R,y]\cap [R,z]$ and $[R,z,y] \leq [R,y]\cap [R,z]$. 

Clearly, the mapping
$f \colon [r,y] \to [r,y,z]$ is a homomorphism from $[R,y]$ to $[R,y,z]$, and since $[R,y]$ has no proper subgroups of finite index, it follows that $[R,y,z]= \{1\}$.  In the same way, we also obtain $[R,z,y]= \{1\}$. Then $[R,y,y]=[R,y,zy]$, and  Lemma \ref{C3} shows that $[R,y]=[R,y,y] \leq [R,zy]$. Similarly, $[R,z] \leq [R,yz]$, where of course  $[R,yz]=[R,zy]$. Therefore $[R,y][R,z] \leq [R,yz]$, and since $m_{yz} \leq m$ and $m_y =m$, it follows that the subgroup $[R,y][R,z]$ is a direct product of exactly $m$ Prüfer subgroups. Moreover, since $[R,y]$ is a direct factor of $[R,y][R,z]$ and it is also a direct product of exactly $m$ Prüfer subgroups, it follows that $[R,z]$ must be finite and so trivial. Since $G$ acts faithfully on $R$, it follows that the element $z$ is trivial, which is a contradiction.
\end{proof}

\begin{lemma}
\label{L12}
Let $G$ be a periodic group, and let $x \in G$ such that $[G,x]$ is a Chernikov group. Let $R$ be the radicable part of $[G,x]$ and assume that $[R,x]=\{1\}$. Then $R=\{1\}$.
\end{lemma}

\begin{proof}
Note that to prove the statement it is enough to show that $R$ is finite. Since $R$ is normal in $G$, it follows that $[R,G,x] = \{1\}$; then applying the Three Subgroup Lemma we have that $R$ is central in $[G,x]$. Therefore the quotient $[G,x]/Z([G,x])$ is finite, and Schur's Theorem yields that $[G,x]'$ is finite. We pass to the quotient $G/[G,x]'$ and assume that $[G,x]$ is abelian. Clearly, there exists a finite subgroup $F$ of $[G,x]$ such that $[G,x]= R \times F$. Put $E= F^G$. Since $E$ is contained in $[G,x]$, it follows that $E$ is abelian and so it has the same exponent as $F$. Therefore $E$ is finite, and hence we can replace $G$ by $G/E$ and assume that $[G,x]=R$. In particular, $[G,x,x]= [R, x] = \{1\}$, and so $\langle x^{G} \rangle$ is abelian.  It follows that the exponent of $\langle x^{G} \rangle$ is exactly the order of $x$, which is finite; then also $[G, x]= R$ has finite exponent, thus $R$ is trivial.
\end{proof}

We will require the following result whose proof can be found in \cite{MES2013}.

\begin{theorem}\label{13}
Let $k$ be a positive integer, and let $G$ be a group such that $\langle g^{X_k(G)}\rangle$ is a Chernikov group for every $g\in G$. Then $\langle g^{\gamma_k(G)}\rangle$ is a Chernikov group for every $g\in G$.
\end{theorem}

We are now ready to prove Theorem \ref{main}. 

\begin{proof}[Proof of Theorem \ref{main}]
In view of Theorem \ref{13}, $\langle g^{\gamma_k(G)}\rangle$ is a Chernikov group for every $g\in G$, and hence 
every subgroup $[\gamma_k(G), g]$ is a Chernikov group as well. For every $g \in G$, let $R_g$ denote the radicable part of $[\gamma_k(G), g]$. An application of Lemma \ref{L14} yields that the subgroup $R$ generated by all the subgroup $R_g$ is a radicable abelian normal subgroup of $G$.

\par
Let $g \in G$; by Lemma \ref{C3} and Lemma \ref{L13}, it follows that $[R, g]=[R, {}_{{k-1}}g]$, where $[R, {}_{{k-1}}g] \subseteq X_k(G)$.  We conclude that
\[[R,g]=[R,g,g] \leq [X_k(G), g] \leq \langle g^{X_k(G)} \rangle,\]
and hence the radicable abelian subgroup $[R,g]$ is contained in the radicable part of $\langle g^{X_k(G)} \rangle$. Therefore, $[R,g]$ is a direct product of at most $m$ Prüfer subgroups.

\par
By the well-known Dedekind's Modular Law, it follows that \[ \langle g^{X_k(G)} \rangle \cap \langle g \rangle[\gamma_k(G),g] =\langle g \rangle(\langle g^{X_k(G)} \rangle \cap [\gamma_k(G),g])= \langle g^{X_k(G)} \rangle,\]
and therefore $\langle g^{X_k(G)} \rangle \cap [g, \gamma_k(G)]$ has finite index in $\langle g^{X_k(G)} \rangle$, as the element $g$ has finite order. In particular, the radicable part of $\langle g^{X_k(G)} \rangle$ is contained in $\langle g^{X_k(G)} \rangle \cap [g, \gamma_k(G)]$, and consequently in $R_g \leq R \cap [g, \gamma_k(G)]$. Then the quotient $\langle g ^{X_k(G)} \rangle / R \cap \langle g ^{X_k(G)} \rangle$ is finite of order at most $n$, so that $\gamma_{k+1}(G)/R$ has finite $(k,n)$-bounded order by [Corollary 1.3, \cite{DDMP2021}].

\par Put $G^*=G/C_G(R)$ and regard  $G^*$ as a group of automorphisms of $R$. Since $R \leq C_G(R)$, it follows that $\gamma_{k+1}(G^*)$ has finite $(k,n)$-bounded order, and hence the centralizer $C^*=C_{G^*}(\gamma_{k+1}(G^*))$ is nilpotent and has finite index $t=t(k,n)$, depending only on $k$ and $n$. \\ On the other hand, Lemma \ref{P8} ensures the existence of an integer $h=h(m)$ such that the subgroup $[R, C^*]$ is a direct product of at most $h$ Prüfer subgroups. Now, let $\{g^*_1, \dots , g^*_t \}$ be a right transversal to $C^*$ in $G^*$. 
We deduce that \[[R,G^*]=[R, C^* \langle g^*_1,\dots,g^*_t \rangle]=[R, C^*][R,g^*_1] \cdots [R,g^*_t], \] and hence $[R,G^*]$ is a direct product of at most $h+tm$ Prüfer subgroups, where the quantity $h+tm$ depends only by $k,m$ and $n$. Observe that $[R,G^*]=[R,G]$.

\par
Finally, we may replace $G$ by the factor group $G/[R,G]$ and assume that $R$ is central. Then, for any $g \in G$, the subgroup $\gamma_k(G)\langle g \rangle$ satisfies the assumption of Lemma \ref{L12}, and therefore we have that $R=\{1\}$.
Consequently, $\gamma_{k+1}(G)$ has finite $(k,n)$-bounded order and the theorem follows.
\end{proof}

\end{document}